\documentclass{lms}
\usepackage{graphicx, amssymb, amsmath, amsfonts} 

\newcommand{\pa}{\partial}

\newcommand{\fp}{\mathfrak{P}}

\newcommand{\cN}{{\mathcal{N}}}

\newcommand{\R}{\mathbb{R}}

\newcommand{\cL}{\mathcal{L}}
\renewcommand{\epsilon}{\varepsilon}

\newtheorem{theorem}{Theorem}[section] 
\newtheorem{corollary}[theorem]{Corollary}

\newnumbered{assertion}{Assertion}    
\newnumbered{conjecture}{Conjecture}  
\newnumbered{definition}{Definition}
\newnumbered{hypothesis}{Hypothesis}
\newnumbered{remark}{Remark}
\newnumbered{note}{Note}
\newnumbered{observation}{Observation}
\newnumbered{problem}{Problem}
\newnumbered{question}{Question}
\newnumbered{algorithm}{Algorithm}
\newnumbered{example}{Example}
\newunnumbered{notation}{Notation} 

\title[One can hear the corners of a drum]
 {One can hear the corners of a drum} 

\author{Z. Lu and J. Rowlett}

\classno{35P99 (primary), 35K05 (secondary).}

\extraline{The first author is partially supported by NSF grant DMS-12-06748.}

\begin{document}
\maketitle

\begin{abstract}
We prove that the presence or absence of corners is spectrally determined in the following sense:  any simply connected domain with piecewise smooth Lipschitz boundary cannot be isospectral to any connected domain, of any genus, which has smooth boundary.  Moreover, we prove that amongst all domains with Lipschitz, piecewise smooth boundary and fixed genus, the presence or absence of corners is uniquely determined by the spectrum.  This means that corners are an elementary geometric spectral invariant; one can hear corners.  
\end{abstract}


\section{Introduction} 
\label{intro}

The isospectral problem is:  if two compact Riemannian manifolds $(M, g)$ and $(M', g')$ have the same Laplace 
spectrum, then are they isometric?  In this generality the answer was first proven to be ``no'' by Milnor \cite{milnor}.  He used a construction of lattices $L_1$ and $L_2$ by Witt \cite{witt} such that the corresponding flat tori $\R^{16}/L_i$ have the same Laplace spectrum for $i=1$, $2$, but are not isometric.  Although the tori are not congruent, the fact that they are isospectral implies a certain amount of rigidity on their geometry.  This rigidity is due to the existence of \em geometric spectral invariants, \em geometric features which are determined by the spectrum.  

The first geometric spectral invariant was discovered by Hermann Weyl in 1912 \cite{weyl}:  the $n$-dimensional volume of a compact $n$-manifold is a spectral invariant.  If the manifold has boundary, then the $n-1$ dimensional volume of the boundary is also a spectral invariant, as shown by Pleijel \cite{pleijel} in 1954.  McKean and Singer \cite{ms} proved that certain curvature integrals are also spectral invariants.  In the particular case of smooth surfaces and smoothly bounded domains, they and M. Kac \cite{kac} independently proved that the Euler characteristic is a spectral invariant.  

M. Kac popularised the isospectral problem for planar domains by paraphrasing it as ``Can one hear the shape of a drum?''  This makes physical sense because if a planar domain is the head of a drum, then the sound the drum makes is determined by the spectrum of the Laplacian with Dirichlet boundary condition.  Therefore, if two such drums are isospectral, meaning they have the same spectrum for the Laplacian with Dirichlet boundary condition, then they sound the same.  Do they necessarily have the same shape?  The negative answer was demonstrated by Gordon, Webb, and Wolpert \cite{gww} who constructed  two polygonal domains which are isospectral but not congruent; see Figure 1.  

 \begin{figure} \label{drum}\includegraphics[width=250pt]{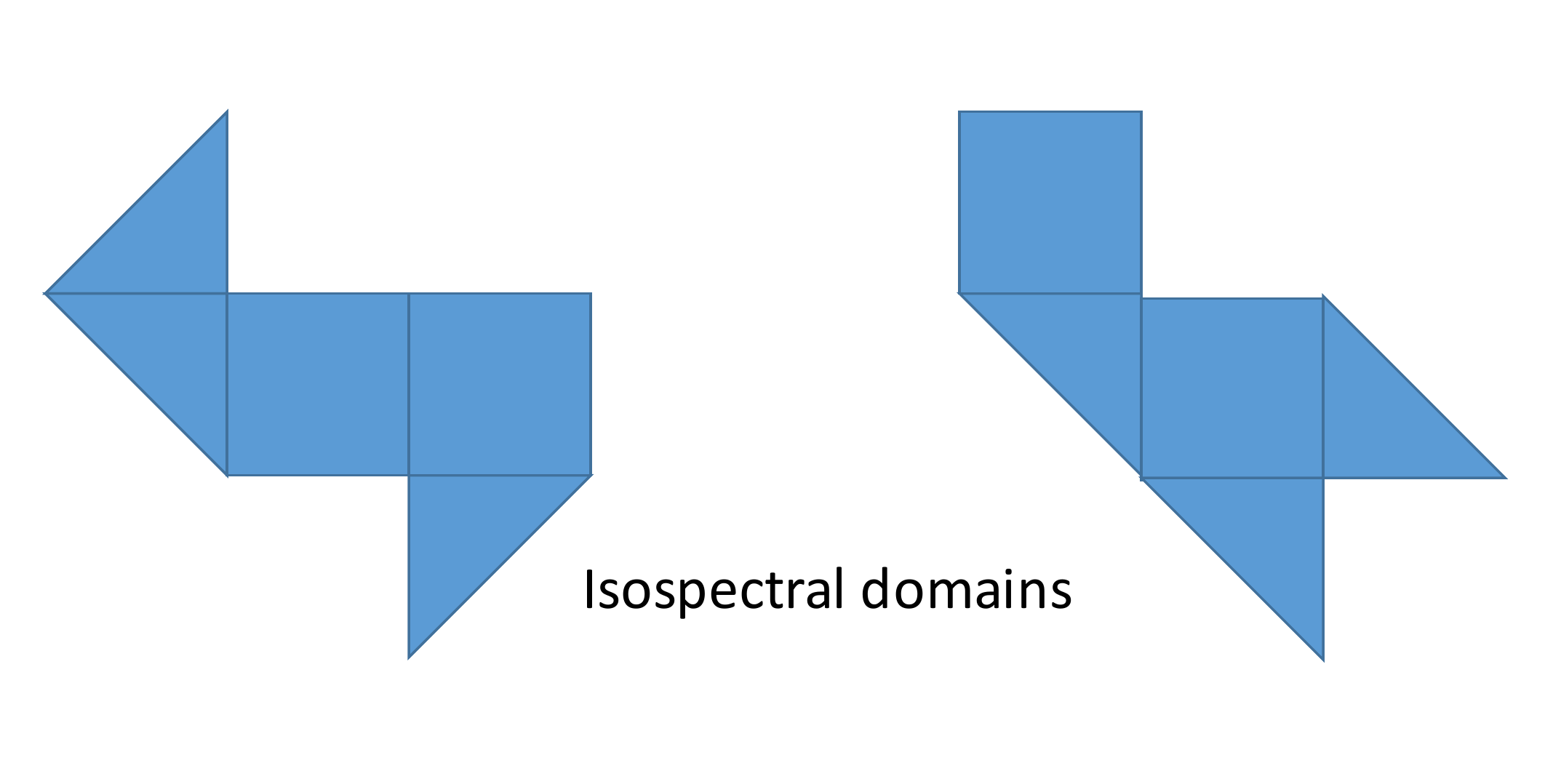} \caption{Identical sounding drums} \end{figure}

On the other hand, the isospectral problem has a positive answer in certain settings.  It is easy to prove that if two rectangular domains are isospectral, then they are congruent.  It requires more work to prove that if two triangular domains are isospectral, then they are congruent.  That was demonstrated in Durso's technically demanding doctoral thesis at MIT who used both the wave and heat traces \cite{durso}.  A simpler proof using only the heat trace was discovered by Grieser and Maronna \cite{gm}.  Zelditch proved \cite{zelditch}   that if two bounded domains with analytic boundary satisfying a certain symmetry condition are isospectral, then they are congruent. 

We prove that the smoothness of the boundary, or equivalently the lack thereof, is a spectral invariant in the following sense:  the spectrum determines whether or not a simply connected domain has corners.  

\begin{theorem} Let $\Omega$ be a simply connected domain with piecewise smooth Lipschitz boundary.  If $\Omega$ has at least one corner, then $\Omega$ is not isospectral to any bounded domain with smooth  boundary which has no corners.  \end{theorem} 
 

In fact, the proof of Theorem 1 implies the following stronger result. 

\begin{corollary} Amongst all planar domains of fixed genus with piecewise smooth Lipschitz boundary, those which have at least one corner are spectrally distinguished.  \end{corollary} 

The corollary means, in the spirit of Kac, that ``one can hear corners.''  Whereas many spectral invariants are complicated abstract objects, there are relatively few known elementary geometric spectral invariants.  Although experts recognize that polygonal domains have an extra, purely local corner contribution to the heat trace, we are not aware of this having been exploited previously to show that corners are spectrally determined.  It therefore appears that we have found the next elementary geometric spectral invariant for simply connected domains since {\cite{kac} and \cite{ms}} showed, half a century ago, that the number of holes is a spectral invariant for smoothly bounded domains.

\section{Proofs}

  \begin{figure}  \includegraphics[width=200pt]{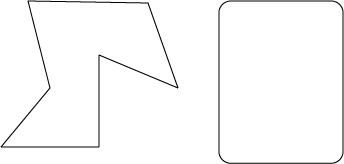} \caption{No polygonal domain is isospectral to a smoothly bounded domain.  One can hear that the left domain has corners, and one can also hear that the right domain has no corners.}   \end{figure} 

\subsection{Preliminaries}
Let $\Omega$ be a bounded domain in the plane.  The Laplace equation with Dirichlet boundary condition is 
\[u_{xx} + u_{yy} = - \lambda u; \quad u|_{\pa \Omega} = 0.\]
The numbers $\lambda$ such that there exists such a solution $u$ are the eigenvalues of the Laplace operator, and the set of all these eigenvalues is the spectrum, which satisfies
\[0 < \lambda_1 < \lambda_2\leq  \cdots \uparrow \infty.\]
The rate at which $\lambda_k \to \infty$ as $k \to \infty$ was determined by Weyl, and is known as Weyl's law \cite{weyl}, 
\[\lim_{k \to \infty} \frac{|\Omega| \lambda_k}{4\pi k} = 1.\]
Above, $|\Omega|$ denotes the area of $\Omega$.  

Positive isospectral results, such as \cite{durso} and \cite{gm},  often employ the heat trace, 
\[h(t) := \sum_{k \geq 1} e^{-\lambda_k t}.\]
Since this quantity is determined by the spectrum, it is a spectral invariant.  It can be shown using Weyl's law that the heat trace is asymptotic to 
\[h(t) \sim \frac{|\Omega|}{4 \pi t} \quad t \downarrow 0.\]
It is also possible to compute the next two terms in the short time asymptotic expansion of the heat trace.  Pleijel proved \cite{pleijel} that  
\[h(t) \sim \frac{|\Omega|}{4 \pi t} - \frac{|\pa \Omega|}{8 \sqrt{\pi t}}, \quad t \downarrow 0.\]
Above $|\pa \Omega|$ denotes the perimeter of $\Omega$.  The next term in the expansion, which we denote by $a_0$, is the key to the proof of our results, and so we briefly explain how to compute it.

As the name suggests, the heat trace is equivalently defined by taking the trace of the heat kernel, that is the fundamental solution of the heat equation.  For an $\cL^2$ orthonormal basis of eigenfunctions $\{\varphi_k\}_{k \geq 1}$ of the Laplacian with corresponding eigenvalues $\{ \lambda_k \}_{k \geq 1}$, the heat kernel is 
$$H(x,y,x', y', t) = \sum_{k \geq 1} e^{-\lambda_k t} \varphi_k (x,y) \varphi_k(x', y').$$
Another way to construct the heat kernel is to begin with an approximate solution to the heat equation, or parametrix, and then solve away the error term using the heat semi-group property and Duhamel's principle; for a manifold without boundary, see for example Chapter 3 of \cite{rose}.  A fundamental property of the heat kernel is that it is asymptotically local for short times. Kac referred this locality principle as ``not feeling the boundary'' \cite{kac}.  He exploited this principle to compute the first three terms in the heat trace expansion as $t \downarrow 0$.  


To show how the third term, $a_0$, in the asymptotic expansion of the heat trace for $t \downarrow 0$  arises, we must also outline the calculation of the first two terms.  The first term comes from considering the most elementary approximation to the heat kernel for a planar domain:  the Euclidean heat kernel on $\R^2$, 
\[\frac{e^{|(x,y) - (x', y')|^2/4t}}{4 \pi t}.\]
By the locality principle, (see also Lemma 2.2 of \cite{htap}) the trace of the actual heat kernel for the domain $\Omega$ is asymptotically equal to the corresponding trace of the Euclidean heat kernel.  Along the diagonal, $(x,y) = (x', y')$, so the Euclidean heat kernel is simply 
\[\frac{1}{4\pi t}.\]
The integral over the domain $\Omega$ is therefore $\frac{|\Omega|}{4 \pi t}$.  Consequently, the leading order asymptotic of the heat trace is 
\[h(t) \sim \frac{|\Omega|}{4 \pi t}, \quad t \downarrow 0.\]

Whereas the Euclidean heat kernel is a good parametrix on the interior, it does not ``feel the boundary.''  Since we assume the boundary is piecewise smooth, in a small neighborhood of each boundary point away from the corner points, the heat kernel for a half plane with Dirichlet boundary condition is a good parametrix, meaning it is a good approximation to the actual heat kernel for the domain.  For the $x \geq 0$ half plane, this heat kernel is 
\[\frac{1}{ 4 \pi t} \left( e^{-\frac{(x - x')^2}{4t}} - e^{-\frac{(x + x')^2}{4t}} \right) e^{-\frac{(y-y')^2}{4t}}.\]

For curved boundary neighborhoods away from any corners, since the boundary is smooth, it is a one-dimensional manifold.  Therefore the smaller the boundary neighborhood is, the closer the boundary is to being straight in this neighborhood.  Consequently, by the locality principle, the trace of the heat kernel for the half plane over this neighborhood  is asymptotically equal to the trace of the heat kernel for $\Omega$ along the same neighborhood.  We calculate that this results in two terms
\[\frac{|\cN|}{4 \pi t} - \frac{|\pa \cN|}{8 \sqrt{\pi t}}.\]
Above $|\cN|$, denotes the area of the neighbourhood, and $|\pa \cN|$ denotes the length of the boundary along this neighbourhood.  Since the first term gives the same local contribution as that of the Euclidean heat kernel, cumulatively we arrive at the first two terms in the asymptotic expansion of the heat trace which was originally demonstrated by Pleijel \cite{pleijel} 
\[h(t) \sim \frac{|\Omega|}{4 \pi t} - \frac{|\pa \Omega|}{8 \sqrt{\pi t}}, \quad t \downarrow 0.\]

In the case of smooth boundary, the next term in the heat trace expansion is best understood by the local parametrix method as in Chapter 3 of \cite{rose}.  McKean and Singer \cite{ms} and Kac \cite{kac} independently proved that the next term is 
\[\frac{1}{12 \pi} \int_{\pa \Omega} k ds, \]
where $k$ is the geodesic curvature of the boundary (see also Proposition 1 of \cite{wat}). By the Gauss-Bonnet Theorem, this integral is equal to $2 \pi \chi(\Omega)$, where $\chi(\Omega)$ is the Euler characteristic of $\Omega$.  So, if there are no corners, then the heat trace 
\[h(t) \sim \frac{|\Omega|}{4 \pi t} - \frac{|\pa \Omega|}{8 \sqrt{\pi t}} + \frac{\chi(\Omega)}{6}, \quad t \downarrow 0.\]
\begin{remark} The difference between $h(t)$ and the above expansion is $O(\sqrt{t})$ as $t \downarrow 0$.  This bound depends only on the $\cL^2$ norm of the curvature of the boundary; see $D_3$ on p. 449 of \cite{wat}.  \end{remark}

For smoothly bounded domains, the number of holes is equal to $1- \chi(\Omega)$, so this shows that the number of holes in the domain is an elementary geometric spectral invariant \em within the class of  smoothly bounded domains.  \em    Interestingly, it is not known whether or not this is a spectral invariant for domains which may have corners.    
The reason is that if there are corners, then these each contribute an extra, purely local term, to the coefficient $a_0$.  This has been called a ``corner defect'' or an ``anomaly'' \cite{htap} due to the fact that the coefficient $a_0$ is not continuous under Lipschitz convergence of domains.

This heat trace coefficient $a_0$ for a domain with polygonal boundary can be computed using the trace of a suitable parametrix in a neighborhood of the corner.  This is precisely what Kac did \cite{kac} using a formula for the heat kernel of a wedge.  His method unfortunately requires that the corner, and the wedge which opens with the same angle as the corner, are convex.  So, the opening angle may only be between $0$ and $\pi$.  Since we do not require convexity, that method will not suffice.  


A more general method follows an idea of D. B. Ray.  Apparently, this calculation was previously performed by Fedosov and published in Russian \cite{fed}.  The idea was to express the Green's function for a wedge, with any opening angle between $0$ and $2\pi$, explicitly as a Kontorovich-Lebedev transform using special functions.  The heat kernel is the inverse Laplace transform of the Green's function, so one can compute its trace using these special functions and integral transformations.  Although Ray's work was never published, and Fedesov's work poses a difficulty to those who cannot read Russian, van den Berg and Srisatkunarajah independently worked through this calculation and published it in \S 2 of \cite{vdbs}.  We have verified their calculation using the reference \cite{gr} by Gradsteyn and Ryzhik on special functions and integral transformations.  

Although van den Berg's and Srisatkunarajah's calculation of the corner contribution to the heat trace was for polygonal boundary, we will show below how to compute for a \em generalized corner.  \em 



\begin{definition} A \em generalized corner \em is a point $p$ on the boundary at which the  following are satisfied.
\begin{enumerate}
\item There exist two tangent vectors based at $p$, $v_1$ and $v_2$, which meet at an angle not equal to $\pi$ such that this angle is the \em opening angle \em at the corner in the following sense.  
\item The boundary in a neighborhood of $p$ is defined by a continuous curve $\gamma(t): (-a, a) \to \R^2$ for $a > 0$ with $\gamma(0) = p$, such that $\gamma$ is smooth {on $(-a,0]$ and $[0,a)$} and  such that 
\[\lim_{t \uparrow 0} \dot \gamma (t) = v_1, \quad \lim_{t \downarrow 0} \dot \gamma  (t) = v_2.\]
\end{enumerate}
\end{definition} 


We use the term  generalized corner, because often a corner is assumed to have straight, non-curved edges in some neighborhood of the corner point.  However, we only require that the boundary curve is asymptotically straight on each side of the corner point.  It is interesting to note that a generalized corner, which may have curved boundary, nevertheless contributes the same purely local term in the heat trace like a classical straight corner. 

  \begin{figure} \caption{A domain with a generalized corner at the point $p$.}  \begin{center}  \includegraphics[width=50pt]{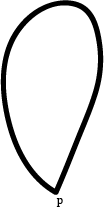} \end{center} \end{figure} 


\begin{theorem} A generalized corner as defined above with opening angle $\theta$ gives a purely local contribution to the heat trace, 
\[\frac{\pi^2 - \theta^2}{24 \pi \theta}.\]
For a bounded domain $\Omega$ with piecewise smooth boundary which may have a finite set of generalized corners, the heat trace 
\[h(t) \sim \frac{ |\Omega|}{4 \pi t} - \frac{| \pa \Omega|}{8 \sqrt{\pi t}} + \frac{1}{12 \pi} \int_{\pa \Omega \setminus \fp} k ds+ \sum_{j=1} ^n \frac{\pi^2 - \theta_j^2}{24 \pi \theta_j}, \quad t \downarrow 0.\]
Above $\fp = \{p_j\}_{j=1} ^n$ is the set of corner points of $\pa \Omega$, with corresponding interior angles $\theta_j$.  
\end{theorem} 

\begin{proof} 
We shall first prove the theorem for the case in which there is only one generalized corner, as illustrated in Figure 3.  Let $\theta$ denote the opening angle at the corner point $p$.  Consider two domains, $\Omega_+$ and $\Omega_-$ such that away from a neighborhood of radius $\epsilon$ about $p$, the boundaries of these domains coincide with that of $\Omega$.  In a neighborhood of radius $\epsilon/2$ about $p$, $\Omega_+$ and $\Omega_-$ each have a corner with straight edges and respective opening angles $\theta_-<\theta<\theta_+$.  The boundaries of $\Omega_\pm$ are smooth except at the point $p$ and such that 
\[\Omega_- \subset \Omega \subset \Omega_+.\]

By domain monotonicity, the eigenvalues of these domains satisfy
\[\lambda_k (\Omega_+) \leq \lambda_k (\Omega) \leq \lambda_k (\Omega_-), \quad \forall k,\]
and therefore the heat traces satisfy 
\begin{equation} \label{ineq1} h_- (t) = \sum_{k \geq 1} e^{-\lambda_k (\Omega_-)t} \leq  \sum_{k \geq 1} e^{- \lambda_k(\Omega) t} \end{equation} 
\begin{equation} \label{ineq2} = h(t) \leq \sum_{k \geq 1} e^{-\lambda_k (\Omega_+) t} = h_+ (t). \end{equation} 

We use the locality principle of the heat kernel to compute the traces.  Let $H_\pm$ denote the heat kernels for $\Omega_\pm$, respectively.  Let $\chi$ be a discontinuous or sharp cut-off function such that in the neighborhood of radius $\epsilon/2$ about the corner point $p$, denoted by $\cN_\epsilon$, $\chi \equiv 1$, and $\chi\equiv 0$ off this neighborhood.  Let $H_{W_\pm}$ be the heat kernels for infinite wedges of opening angles $\theta_\pm$, respectively.  By Lemma 2.2 of \cite{htap}, the heat traces satisfy
\[ \left| h_\pm(t) - Tr (\chi H_{W_\pm} + (1-\chi) H_\pm)(t) \right| = O(t^{\infty}), \quad t \downarrow 0. \]

In \S 2 of \cite{vdbs} the trace of the heat kernel for an infinite wedge was computed over a finite wedge with the same opening angle as the infinite one.  This trace is
\begin{equation} \label{wedge} \frac{|W|}{4 \pi t} - \frac{|\pa W|}{8 \sqrt{\pi t}} + \frac{\pi^2 - \theta^2}{24 \pi \theta} + O(e^{-c/t}), \end{equation} 
where $|W|$ and $|\pa W|$ denote the area and the length of the sides of the (finite) wedge, and $\theta$ is the opening angle.  Theorem 1 of \cite{vdbs} gives an estimate for the positive constant $c$.

Combining (\ref{wedge}) with the trace of $H_\pm$ away from the corner, as computed in \cite{ms} (see also \cite{wat}), we have 
\[Tr (\chi H_{W_\pm} + (1-\chi) H_\pm)(t) = \frac{|\Omega_\pm|}{4 \pi t} - \frac{|\pa \Omega_\pm|}{8 \sqrt{\pi t}} + \frac{\pi^2 - \theta_\pm ^2}{24 \pi \theta_\pm} \]
\[+ \frac{1}{12 \pi} \int_{\pa \Omega_\pm \setminus \cN_\epsilon} k ds + O(\sqrt{t}). \]
 Therefore, 
 \begin{equation} \label{trace4} h_\pm (t) = \frac{|\Omega_\pm|}{4 \pi t} - \frac{|\pa \Omega_\pm|}{8 \sqrt{\pi t}} + \frac{\pi^2 - \theta_\pm^2}{24 \pi \theta_\pm}  + \frac{1}{12 \pi} \int_{\pa \Omega_\pm \setminus \cN_\epsilon} k ds+ O(\sqrt{t}). \end{equation} 
 
 
As  $\epsilon \downarrow 0$, 
\[\frac{|\Omega_\pm|}{4 \pi t} \to \frac{|\Omega|}{4 \pi t}, \quad \frac{|\pa \Omega_\pm|}{8\sqrt{ \pi t}} \to \frac{|\pa \Omega|}{8 \sqrt{\pi t}}, \quad \frac{\pi^2 - \theta_\pm^2}{24 \pi \theta_i} \to \frac{\pi^2 - \theta^2}{24 \pi \theta},\]
\[\int_{\pa \Omega_\pm \setminus \cN_\epsilon} k ds \to \int_{\pa \Omega \setminus \{p\}} k ds.\] 
By Remark 1, the remainders in (\ref{trace4}) are uniformly $O(\sqrt{t})$ as $t \downarrow 0$, due to the boundedness of the curvature of the boundary.  Therefore, the convergence given above, together with (\ref{ineq1}), (\ref{ineq2}), and (\ref{trace4}), show that 
\[h(t) = \frac{|\Omega|}{4 \pi t} - \frac{|\pa \Omega|}{8 \sqrt{\pi t}} + \frac{\pi^2 - \theta^2}{24 \pi \theta} + \frac{1}{12\pi} \int_{\pa \Omega \setminus \{p\}} k ds + O(\sqrt{t}), \quad t \downarrow 0.\]
The same argument applied at each corner together with the locality principle of the heat kernel gives the result in case the domain has a finite number of corners.   
\end{proof}


If $\pa \Omega$ has corners, then the curvature integral 
\begin{equation} \label{curv-red1} \int_{\pa \Omega \setminus \fp} k ds = \sum_{j=1} ^n \theta_j + \pi(2 \chi(\Omega) - n), \end{equation} 
and therefore 
\begin{equation} \label{red} a_0 = \frac{1}{12}(2\chi(\Omega) - n) + \sum_{j=1} ^n \frac{\pi^2 + \theta_j^2}{24 \pi \theta_j},\end{equation} 
If $\Omega$ has no corners, then the curvature integral is $2 \pi \chi(\Omega)$, and therefore
\[a_0 = \frac{\chi(\Omega)}{6}.  \]
We shall now proceed with the proof of Theorem 1.  

\begin{proof}[of Theorem 1] 
For a domain with at least one corner, the third heat trace invariant $a_0$ is given by (\ref{red}).
For a domain $\Omega'$ with no corners, then the curvature integral is 
\[2 \pi \chi(\Omega') \implies a_0 \leq \frac{1}{6}.\]

Since the domain with corners $\Omega$ is simply connected, 
\[a_0 = \frac{1}{24} \sum_{k=1} ^n \left( \frac{1}{x_k} + x_k \right) - \frac{n}{12} + \frac{1}{6}, \quad x_k = \frac{\theta_k}{\pi}.\]

Consider the function 
\[f: \R^n \to \R, \quad f(x) = \sum_{k=1} ^n \left( \frac{1}{x_k} + x_k \right).\]
If any $x_k \downarrow 0$, then $f \uparrow \infty$, so $f$ does not have a maximum.  The gradient of $f$ vanishes precisely when $x_k =1$ for all $k$.  This is therefore the global minimum of the function $f$ restricted to $x \in \R^n$ such that $x=(x_1, \ldots, x_n)$ with all $x_k > 0$.  

Note that $x_k=1$ corresponds to an interior angle is equal to $\pi$ at the $k^{th}$ corner, which is impossible if there is indeed a corner there.  Consequently, we cannot have $x_k = 1$ for all $k$, and so 
\[f(x) > f(1, \ldots, 1) = 2n.\]
We therefore have
\[a_0 = \frac{1}{24} f(x) - \frac{n}{12} + \frac{1}{6} > \frac{2n}{24} - \frac{1}{12} + \frac{1}{6} > \frac{1}{6}.\]
 \end{proof}

An immediate consequence of our proof is the following.  

\begin{corollary} One can hear corners.  Specifically, a bounded simply connected domain $\Omega \subset \R^2$ with piecewise smooth boundary has at least one corner if and only if 
\[a_0 > \frac{1}{6}.\]
Equivalently, a bounded domain $\Omega \subset \R^2$ has smooth boundary if and only if 
\[a_0 \leq \frac{1}{6}.\]
\end{corollary} 

We conclude with the proof of Corollary 1.  
\begin{proof}[of Corollary 1] 
For a connected bounded domain with piecewise smooth boundary and at least one corner, then as above we have $a_0$ given by (\ref{red}), which we can equivalently express as  
\[a_0 = \frac{1}{24} \sum_{k=1} ^n (x_k ^{-1} + x_k)  - \frac{n}{12}+  \frac{\chi(\Omega)}{6},  \quad \theta_k = \pi x_k.\]
It therefore follows that if the Euler characteristic is fixed, by the above arguments, 
\[a_0 > \frac{\chi(\Omega)}{6}.\]
Since for any smoothly bounded domain $\Omega'$ with Euler characteristic identical to that of $\Omega$, we have
\[a_0 (\Omega') = \frac{\chi(\Omega')}{6} = \frac{\chi(\Omega)}{6} < a_0 (\Omega)\]
the result follows.  
\end{proof}

\affiliationone{
  Z. Lu \\
Department of Mathematics \\
University of California, Irvine\\
Irvine, CA 92697 USA \\
   \email{zlu@uci.edu}}

\affiliationone{
   J. Rowlett\\
  Chalmers University of Technology \\ Mathematical Sciences \\ Gothenburg 41296 Sweden 
   \email{julie.rowlett@chalmers.se}}
\end{document}